\documentclass[a4paper]{amsart}

\pdfoutput=1

\usepackage[T1]{fontenc}
\usepackage[utf8]{inputenc}
\usepackage{lmodern}
\usepackage{microtype}
\usepackage{amsmath,amssymb,braket,tikz}
\usetikzlibrary{tikzmark,calc}

\usepackage{url}
\usepackage{hyperref}
\usepackage{tikz}
\usepackage{tikz-cd} 
\usepackage{array}
\usepackage[english]{babel}
\usepackage{amsthm}
\usepackage{amsmath}
\usepackage{ dsfont }
\usepackage{caption}
\usepackage{amsfonts}
\usepackage{epigraph}
\usepackage{amssymb}
\usepackage{blindtext}
\usepackage{subcaption}
\usepackage{caption}
\usepackage{tabu}

\usepackage{multicol, latexsym}
\usepackage{blindtext}
\usepackage{mdframed}
\usepackage{ bbold }
\usepackage{afterpage}
\usepackage{comment}
\usetikzlibrary{positioning,shadows,arrows}

\usepackage{tikz}
\usetikzlibrary{trees}
\usepackage{float}
\newcommand{\R}{\mathds{R}}
\newcommand{\Q}{\mathds{Q}}
\newcommand{\Z}{\mathds{Z}}

\newcommand{\N}{\mathds{N}}

\newcommand{\cd}{\text{cd}}
\newcommand{\core}{\operatorname{core}}

\newtheorem*{theorem*}{Theorem}

\usepackage[inner=3cm,outer=3cm,top=3.3cm,bottom=3cm]{geometry}

\usepackage{amssymb, amsfonts}
\usepackage{mathtools}

\usepackage{paralist}

\usepackage{xcolor}

\usepackage[english]{babel} 
\usepackage[autostyle]{csquotes} 
\usepackage{tikz, tikz-cd}

\theoremstyle{plain}
\newtheorem{theorem}{Theorem}[section]
\newtheorem{prop}[theorem]{Proposition}
\newtheorem{corollary}[theorem]{Corollary}
\newtheorem{lemma}[theorem]{Lemma}
\newtheorem{conjecture}[theorem]{Conjecture}

\theoremstyle{definition}
\newtheorem{definition}[theorem]{Definition}
\newtheorem{example}[theorem]{Example}
\newtheorem{remark}[theorem]{Remark}

\newcommand{\rleft}{\mathopen{}\mathclose\bgroup\left}
\newcommand{\rright}{\aftergroup\egroup\right}

\newcommand{\st}{\: | \:}

\title{Fine Polyhedral Adjunction Theory}

\author{Sofía Garzón Mora}
\address{Mathematik, Freie Universität Berlin, 14195 Berlin, Germany.}
\email{sofia.garzon.mora@fu-berlin.de}

\author{Christian Haase}
\address{Mathematik, Freie Universität Berlin, 14195 Berlin, Germany.}
\curraddr{}
\email{haase@math.fu-berlin.de}


\begin{document}
\selectlanguage{english}

\begin{abstract}
Originally introduced by Fine and Reid in the study of plurigenera of toric hypersurfaces \cite{Fin83,Re85}, the Fine interior of a lattice polytope got recently into the focus of research. It is has been used for constructing canonical models in the sense of Mori Theory \cite{Bat20}. Based on the Fine interior, we propose here a modification of the original adjoint polytopes as defined in \cite{DiR14}, by defining the Fine adjoint polytope $P^{F(s)}$ of $P$ as consisting of the points in $P$ that have lattice distance at least $s$ to all valid inequalities for $P$. We obtain a Fine Polyhedral Adjunction Theory that is, in many respects, better behaved than its original analogue. Many existing results in Polyhedral Adjunction Theory carry over, some with stronger conclusions, as decomposing polytopes into Cayley sums, and most with simpler, more natural proofs as in the case of the finiteness of the Fine spectrum.
\end{abstract}

\maketitle{}

\section{Introduction}
\label{sec:intro}
Let $P \subset \R^d$ be a rational polytope, i.e., the convex hull of a finite set points in $\Q^n$.
For $s>0$ we define the \textit{Fine adjoint polytope} $P^{F(s)}$ as the set of points satisfying $\langle a,x \rangle \ge b+s$ whenever $\langle a,x \rangle \ge b$ is valid for all points in $P$. 
The Fine interior of a polytope, which refers to the Fine adjoint polytope where $s=1$, was firstly introduced by Jonathan Fine in \cite{Fin83}, where it was referred to as the heart of a polytope. Along the lines of this idea, we take here the Fine adjoint polytope as consisting of the points of $P$ with lattice distance at least $s$ to \textit{all valid inequalities} for $P$ (cf.~Definition~\ref{def: fineadj}).
In this paper we argue that the resulting Fine Polyhedral Adjunction Theory is a more natural version of the Polyhedral Adjunction Theory introduced in~\cite{DiR14} where only facet defining inequalities were considered. 
Various of the results obtained from the original Polyhedral Adjunction Theory carry over to the Fine case. Some of them give us stronger conclusions, such as the Decomposition Theorem~\ref{Thm: decomposition}. Moreover, we are often able to provide simpler, more elegant proofs, for example, as in the case of the finiteness of the Fine spectrum as mentioned in the Theorem~\ref{thm:finitespectrum}.\\

\begin{figure}[H]
    \centering
    \begin{minipage}{.3\textwidth}
        \centering
        \begin{tikzpicture}[scale=0.9]
\draw (0,0) -- (5,0) -- (5,2) -- (3,4) -- (0,4) -- (0,0);
\draw (1/2,1/2) -- (9/2,1/2) -- (9/2,2) -- (3,7/2) -- (1/2,7/2) -- (1/2,1/2);
\draw (1,1) -- (4,1) -- (4,2) -- (3,3) -- (1,3) -- (1,1);
\draw (3/2,3/2) -- (7/2,3/2) -- (7/2,2) -- (3,5/2) -- (3/2,5/2) -- (3/2,3/2);
\draw (2,2) -- (3,2);
\end{tikzpicture}
    \end{minipage}%
    \begin{minipage}{0.3\textwidth}
\centering
\begin{tikzpicture}[scale=0.9]
\draw (2,0) -- (1,2) -- (-1,2) -- (-2,0) -- (-1,-2) -- (1,-2) -- (2,0);
\draw (-1,3/2) -- (-3/2,1/2) -- (-3/2,-1/2) -- (-1,-3/2) -- (1,-3/2) -- (3/2,-1/2) -- (3/2,1/2) -- (1,3/2) -- (-1,3/2);
\draw (1,1) -- (-1,1) -- (-1,-1) -- (1,-1) -- (1,1);
\draw (1/2,1/2) -- (-1/2,1/2) -- (-1/2,-1/2) -- (1/2,-1/2) -- (1/2,1/2);
    
\filldraw[black] (0,0) circle (2pt);
    \end{tikzpicture}
    \end{minipage}
\begin{minipage}{0.3\textwidth}
        \centering
        \begin{tikzpicture}[scale=0.9]
\draw (0,0) -- (5,0) -- (1,4) -- (0,4) -- (0,0);
\draw (1/2,1/2) -- (4,1/2) -- (1,7/2) -- (1/2,7/2) -- (1/2,1/2);
\draw (1,1) -- (3,1) -- (1,3) -- (1,1);
\draw (7/3,4/3) -- (4/3,4/3) -- (4/3,7/3) -- (7/3,4/3);
\filldraw[black] (5/3,5/3) circle (2pt);
    \end{tikzpicture}
    \end{minipage}
    \caption{Examples of polytopes with their Fine adjoint polytopes.}
    \label{fig:adjoints}
\end{figure}
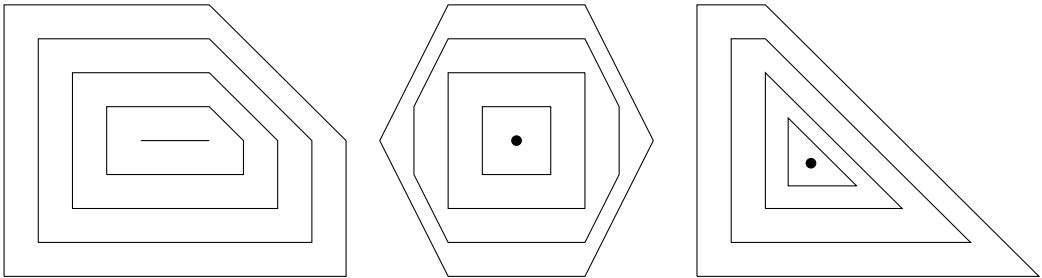

Of particular importance is the largest value $s_0$ of $s$ so that $P^{F(s)}$ is non-empty. For historical reasons, we record its reciprocal $\mu^F := 1/s_0$.
A first indication that the Fine theory is better behaved is the monotonicity of this parameter:
$$ P \subseteq Q \qquad \Longrightarrow \quad \mu^F(P) \ge \mu^F(Q),$$
which does not hold for the analogous parameter $\mu$, as defined in \cite{DiR14}.\\

One of the main results of Polyhedral Adjunction Theory is the decomposition theorem (cf.~\cite{DiR14, DN10, HBP09}). Here, we refer to a \textit{Cayley sum} $P_0 \star \cdots \star P_t$ of $t+1$ polytopes as being a polytope which is constructed by locating the $t+1$ polytopes along the vertices of a $t$-dimensional standard simplex and taking its convex hull (as in Definition \ref{Def: Cayleysum}). A conjecture posed by Dickenstein and Nill \cite[Conj. 1.2]{DN10} about the Cayley decomposition of $n$-dimensional polytopes with codegree bounded below by $\frac{n+3}{2}$ has been disproven by Higashitani \cite{Hi19}. Instead a weaker version was proposed \cite[Conj. 1.3]{DiR14} which states the following.

\begin{conjecture} \label{Conj: qcod}
If an $n$-dimensional lattice polytope $P$ satisfies $\mu(P) > \frac{n+1}{2}$, then $P$ decomposes as a Cayley sum of lattice polytopes of dimension at most $\lfloor 2(n + 1 - \mu(P)) \rfloor$.
\end{conjecture}

This conjecture is still open but a slightly weaker version was proven in \cite[Thm. 3.4]{DiR14} in which $\mu(P) > \frac{n+1}{2}$ is replaced with just $\mu(P) \geq \frac{n+2}{2}$. Because of the way the Fine adjoint polytopes are defined, we obtain the relation of the $\Q$-codegree and Fine $\Q$-codegree of a rational polytope $P$, $\mu^F(P)$, to be such that
\begin{equation} \label{eq: mufinegeq}
    \mu(P) \leq \mu^F(P).
\end{equation}
It is due to this that in Theorem~\ref{Thm: decomposition} we prove a Fine version of this decomposition theorem where essentially the same proof yields a stronger result.\\

Our second main result is related to Fujita's Spectrum Conjecture.

\begin{conjecture}[Spectrum Conjecture, Fujita \cite{Fu96}]
For any $n \in \Z_{\geq 1}$, let $S_n$ be the set of unnormalized spectral values of a smooth polarized $n$-fold. Then, for any $\varepsilon > 0$, the set $\{\mu \in S_n \st \mu > \varepsilon \}$ is a finite set of rational numbers.
\end{conjecture}

A polyhedral version of Fujita's conjecture was proven by Paffenholz~\cite[Thm 3.1]{Paf12} even allowing certain, $\alpha$-canonical, singularities (cf. Theorem \ref{Thm: finiteqcodspec} below).
In this paper, we show that the analogous set 
$\{\mu^F \in S^F_n \st \mu^F > \varepsilon \}$ of Fine spectral values is finite without any assumption on the singularities (cf.~Theorem~\ref{thm:finitespectrum}).
As a result, our proof is simpler than Paffenholz', and it should allow for classification results in the future.

\section*{Acknowledgments}
\noindent
The authors would like to thank Benjamin Nill for his insightful comments, helpful reviews and dedication to this project.
The first author was supported by the Deutsche Forschungsgemeinschaft (DFG), Graduiertenkolleg ``Facets of Complexity'' (GRK 2434).

\section{Redefining Polyhedral Adjunction Theory}
\label{sec:polyhedraladjunction}
\noindent
In what follows, unless stated otherwise, we consider $P \subseteq \R^n$ to be an $n$-dimensional rational polytope, which is described in a unique minimal way by inequalities as $P = \{x \in \R^n \st \langle a_i, x \rangle \geq b_i, i=1,...,m\},$ where $b_i \in \Q$ and $a_i \in (\Z^n)^*$ are the primitive rows of a matrix $A$, i.e., they are not the multiple of another lattice vector, and $b \in \Q^m$. We will refer to $P$ being a \textit{rational polytope} as having its vertices lie in $\Q^n$ and we will say that $P$ is a \textit{lattice polytope} if its vertices lie in $\Z^n$. We introduce our first definitions.

\begin{definition}
Let $f$ be the affine functional $f(x) = \langle a , x \rangle -b$ for some $b \in \Q$ and $a \in (\Z^n)^*$. Such a functional is said to be \textit{valid} for a polytope $P$ if for the halfspace $\mathcal{H}_+ := \{x \in \R^n \st f(x) \geq 0\},$ we have that $P \subseteq \mathcal{H}_+$. 
Moreover, if there is some $p \in P$ with $f(p)=0$, i.e., at least one point of $P$ lies in $\mathcal{H}$, the hyperplane generated by $f$, we say $f$ is a \textit{tight} valid inequality for $P$.
\end{definition}

Since $P$ is a polyhedron, note that it can be described by a finite subset of all tight valid inequalities for $P$ of which there is an infinite number, namely at least one for each primitive $a \in (\Z^n)^* \setminus \{0\}$.

\begin{definition} \label{def: fineadj}
Let $\alpha \in (\R^n)^*$ we define the \textit{distance} function associated with $P$ as
$$d_P^F : (\R^n)^* \to \R, \quad \alpha \mapsto \min_{x \in P} \langle \alpha, x \rangle.$$
In terms of this function, for some real number $s > 0$, we may define the \textit{Fine adjoint polytope}, which is a rational polytope, as
$$P^{F(s)} := \{ x \in \R^n \st \langle a , x \rangle \geq d^F_P(a) + s, \text{ for all } a \in (\Z^n)^* \setminus \{0\} \}.$$
We will refer to the study of such Fine adjoint polytopes as \textit{Fine polyhedral adjunction theory}.
\end{definition}

As previously mentioned, the Fine adjoint polytopes we have introduced are a variant of the adjoint polytopes as defined in \cite[Definition 1.1]{DiR14}. In order to compare between these definitions, we recall the original one here.

\begin{definition} \label{def: classical mu}
Let $P$ be a rational polytope of dimension $n$ given by the inequalities $\langle a_i, \cdot \rangle \geq b_i $ for $i=1,...,m$ that define facets $F_1,...,F_m$ in a minimal way. Then for $x \in \R^n$, the \textit{lattice distance} from the facet $F_i$ is given by
$$d_{F_i} := \langle a_i , x \rangle -b_i $$
and the \textit{lattice distance with respect to the boundary} $\partial P$ of $P$ is
$$d_{P} := \min_{i=1,...,m} d_{F_i}(x).$$
For $s>0$, the \textit{adjoint polytope} is defined as
$$P^{(s)}:= \{x \in \R^n \st d_P(x) \geq s\}.$$
\end{definition}

\begin{remark}
In some cases, taking Fine adjoint polytopes of a polytope $P$ will be the equivalent to considering the original adjoint polytopes, as is the case of the rightmost and leftmost examples in Figure \ref{fig:adjoints}.
\end{remark}

In what follows, we will prove a crucial result, namely that only finitely many tight valid inequalities $f_1,...,f_t$ will be relevant when computing the Fine adjoint polytopes. Moreover, from its proof we will obtain a characterization for when exactly an inequality will be relevant for computing the Fine adjoint polytopes.
We make this notion of a relevant inequality more precise.

\begin{definition}
Let $\mathcal{F}$ be the set of all valid inequalities for $P$, where an element $f \in \mathcal{F}$ is of the form $\langle a_f , x \rangle \geq b_f$. A valid inequality $f \in \mathcal{F}$ is said to be \textit{relevant} for $P$ if for some $s > 0$, it holds that
$$\{x \in \R^n \st \langle a_f , x \rangle \geq d^F_P (a_f) + s \text{ } \forall f \in \mathcal{F} \} \neq \{ x \in \R^n \st \langle a_f, x \rangle \geq d^F_P (a_f) + s \text{ }\forall f \in \mathcal{F} \setminus \{f\}\}.$$
The valid inequality $f$ is said to be \textit{irrelevant} if it is not relevant.
\end{definition}

The following proposition will be very useful for our results and computations below.

\begin{prop}[{\cite[Proposition 3.11]{Bat20}}] \label{thm:finitesupport}
Let $P$ be a rational polytope of dimension $n$. Then there exists a finite set $\mathcal{S} \subset \mathcal{F}$ of valid inequalities for $P$ such that the set $\mathcal{S}$ contains all relevant valid inequalities for $P$.
\end{prop}

From Proposition \ref{thm:finitesupport} we obtain a useful description of the relevant valid inequalities which we state as a Corollary.

\begin{corollary} \label{cor: relevantineqs}
Let $P$ be a rational polytope of dimension $n$. The valid relevant inequalities for $P$ of the form $\langle a, \cdot \rangle \geq d_P^F(a)$ correspond to the $a \in (\mathbb{Z}^n)^*$ such that $a \in \operatorname{conv}( a_1,...,a_m)$ where the $a_i$ for $1 \leq i \leq m$ are the primitive inward pointing facet normals of $P$. 
\end{corollary}

We will now consider the polytope $P$ to be defined as
\begin{equation} \label{eq: minimaldescription}
    P = \{x \in \R^n \st \langle a_i, x \rangle \geq b_i, i=1,...,m\},
\end{equation}
where $b_i \in \Q$ and $a_i \in (\Z^n)^*$ are the primitive rows of a matrix $A$ including all relevant valid inequalities for $P$.\\

\begin{remark}
Note that in the Fine case, we have that taking Fine adjoint polytopes satisfies monotonicity with respect to inclusion of polytopes, i.e., if $P$ and $Q$ are two polytopes, such that $P \subseteq Q,$ then we have that for any $s \geq 0$, $P^{F(s)} \subseteq Q^{F(s)}$. This holds since for any $a \in (\Z^n)^*$ it follows that $d_P^F(a) \geq d_{Q}^F(a)$, but this does not necessarily hold in the original polyhedral adjunction case. 
\end{remark}

Now, using the Fine adjoint polytopes, we may reformulate the concept of the $\Q-$codegree.

\begin{definition}
The \textit{Fine $\Q$-codegree} of a rational polytope $P$ is 
$$\mu^{F}(P) := (\sup \{ s > 0 \st P^{F(s)} \neq \emptyset \} )^{-1},$$
and the \textit{Fine core} of $P$ is 
$$\core^F(P) := P^{(1/\mu^{F}(P))}.$$
\end{definition}

\begin{example}
In general, the core and the Fine core of a given polytope can vary and may even have different dimensions, as in the case of the polytopes in Figure \ref{fig:diffcore}.

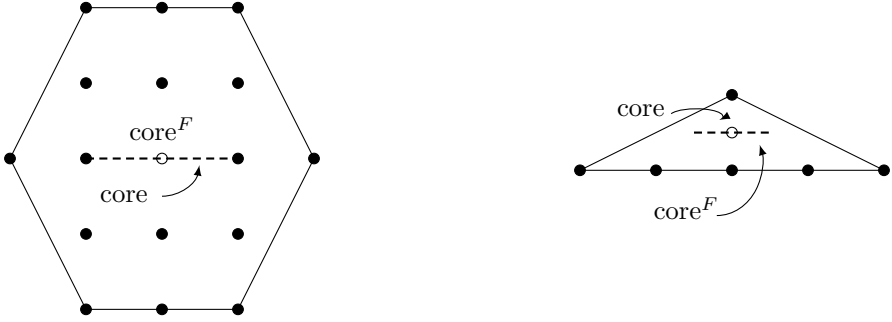
\begin{figure}[H]
    \centering
    \begin{minipage}{.5\textwidth}
        \centering
        \begin{tikzpicture}

\draw (2,0) -- (1,2) -- (-1,2) -- (-2,0) -- (-1,-2) -- (1,-2) -- (2,0);
\node[label={$\core^F$}] at (0,0) {};
\node (core) at (-1/2,-1/2) {$\core$};

\draw[-latex] ($(0,-.5)$) arc
    [
        start angle=-90,
        end angle=-10,
        x radius=0.5cm,
        y radius =0.5cm
    ] ;
    
\draw[black] (0,0) circle (2pt);
\filldraw[black] (0,1) circle (2pt);
\filldraw[black] (0,2) circle (2pt);
\filldraw[black] (0,-1) circle (2pt);
\filldraw[black] (0,-2) circle (2pt);
\filldraw[black] (1,-1) circle (2pt);
\filldraw[black] (1,0) circle (2pt);
\filldraw[black] (1,1) circle (2pt);
\filldraw[black] (-1,-1) circle (2pt);
\filldraw[black] (-1,0) circle (2pt);
\filldraw[black] (-1,1) circle (2pt);
\filldraw[black] (-1,2) circle (2pt);
\filldraw[black] (1,2) circle (2pt);
\filldraw[black] (-1,-2) circle (2pt);
\filldraw[black] (1,-2) circle (2pt);
\filldraw[black] (2,0) circle (2pt);
\filldraw[black] (-2,0) circle (2pt);

\draw[thick, densely dashed] (-1,0) -- (1,0);

\end{tikzpicture}
    \end{minipage}%
    \begin{minipage}{0.5\textwidth}
        \centering
        \begin{tikzpicture}
\draw (-2,0) -- (2,0) -- (0,1) -- (-2,0);
\draw[thick,densely dashed] (-1/2,1/2) -- (1/2,1/2);

\draw[black] (0,1/2) circle (2pt);

\filldraw[black] (-2,0) circle (2pt);
\filldraw[black] (-1,0) circle (2pt);
\filldraw[black] (0,0) circle (2pt);
\filldraw[black] (1,0) circle (2pt);
\filldraw[black] (2,0) circle (2pt);
\filldraw[black] (0,1) circle (2pt);

\node (coreF) at (-.6,-2/4) {$\core^F$};

\draw[-latex] ($(-0.2,-3/5)$) arc
    [
        start angle=-90,
        end angle=20,
        x radius=0.6cm,
        y radius =0.7cm
    ] ;
    
\node (core) at (-1.2,0.8) {$\core$};

\draw[-latex] ($(-.8,0.75)$) arc
    [
        start angle=160,
        end angle= -30,
        x radius=0.35cm,
        y radius =0.1cm
    ] ;

    \end{tikzpicture}
    \end{minipage}
    \caption{Examples of polytopes whose Fine and original cores differ.}
    \label{fig:diffcore}
\end{figure}

\end{example}

\begin{example} \label{Ex: finecorenotcontained}
Consider the polytope given as in \cite[Figure 6]{DiR14} for the case $h=10$ by 
$$P = \operatorname{conv} \begin{bmatrix}
    0 & 2 & 0 & 2 & 0 & 0 \\
    0 & 0 & 4 & 2 & 0 & 4 \\
    0 & 0 & 0 & 0 & 10 & 10 \\
\end{bmatrix}.$$
We may compute the core of $P$ and the Fine core of $P$ to be 
$$\core(P) = \operatorname{conv} \begin{bmatrix}
    4/3 & 4/3 \\
    4/3 & 4/3 \\
    4/3 & 2 \\
\end{bmatrix}, \quad\quad \core^F(P) = \operatorname{conv} \begin{bmatrix}
    1 & 1 & 1 & 1\\
    1 & 1 & 2 & 2\\
    1 & 3 & 1 & 3\\
\end{bmatrix}.$$
Thus, in this case, as seen in Figure \ref{fig:finecoredifferent}, we have that $\core^F(P) \nsubseteq \core(P)$. The core and the Fine core are even disjoint for this example.

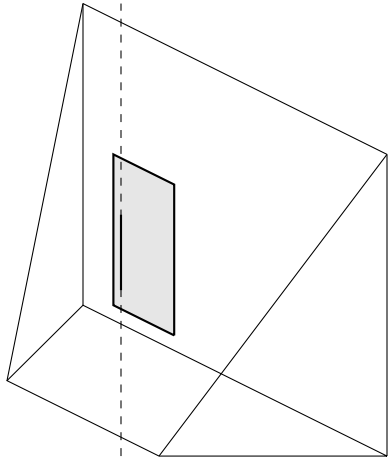
\begin{figure}[H] 
    \centering
    \begin{tikzpicture}
    \draw (4,-2) -- (0,0) -- (-1,-1) -- (1,-2) -- (4,-2); 
    \draw (0,0) -- (0,4);
    \draw (0,4) -- (-1,-1);
    \draw (4,-2) -- (4,2);
    \draw (4,2) -- (0,4);
    \draw (1,-2) -- (4,2);

    \draw[dashed] (0.5,-2) -- (0.5,4);
    \draw[thick] (0.5,0.2) -- (0.5,1.2);

    \filldraw[thick, draw=black, fill opacity=0.1] (1.2,-0.4) -- (0.4,0) -- (0.4,2) -- (1.2,1.6) -- (1.2,-0.4);
    \end{tikzpicture}
    \caption{Polytope whose Fine core is not contained in its classical core.}
    \label{fig:finecoredifferent}
\end{figure}
\end{example}

\noindent
Let us now denote the \textit{normal fan} of a polytope $P$ by $\mathcal{N}(P)$. We will use this notion to define a second invariant for the Fine adjoint polytopes.

\begin{definition}
The \textit{Fine nef value} of a rational polytope $P$ is 
$$\tau^F(P) := (\sup \{s > 0 \st \mathcal{N}(P^{F(s)}) = \mathcal{N}(P)\})^{-1} \in \R_{>0} \cup \{ \infty \}.$$
\end{definition}

\noindent
As opposed to the case of the Fine $\Q$-codegree, here the supremum may not be the maximum.

\begin{definition} \label{Def: gorenstein}
Let $\sigma \subset (\R^n)^*$ be an $n$-dimensional rational polyhedral cone with primitive generators $a_1,...,a_k$. The cone $\sigma$ is called \textit{$\Q$-Gorenstein of index $r_{\sigma}$} if there is a primitive vector $u_{\sigma}$ such that $\langle a_i ,u_{\sigma} \rangle = r_{\sigma}$ for $1 \leq i \leq k$. If $r=1$, the cone $\sigma$ is called \textit{Gorenstein}. We say that a complete rational polyhedral fan $\Sigma$ is \textit{$\Q$-Gorenstein of index r}, resp. \textit{Gorenstein}, if the maximal cones $\sigma \in \Sigma$ are $\Q$-Gorenstein of index $r_{\sigma}$ and $r = \operatorname{lcm}(r_{\sigma} \st \sigma \in \Sigma)$, resp. Gorenstein.
\end{definition}

\begin{definition} \label{def: alphacanonical}
\noindent
If we consider an element $y$ of a $k$-dimensional rational polyhedral cone $\sigma$ with generators $a_1,...,a_k$, then the \textit{height function} associated with the cone $\sigma$ is the piecewise linear function given by

$$\operatorname{height}_\sigma(y) := \max \Bigg\{ \sum_{i=1}^k \lambda_i \text{ } \Bigg| \text{ } y = \sum_{i=1}^k \lambda_i a_i, \text{ and } \lambda_i \geq 0 \text{ for } 1 \leq i \leq k \Bigg \}.$$

\noindent
The cone $\sigma$ is \textit{$\alpha$-canonical} for some $\alpha > 0$ if $\operatorname{height}_\sigma(y) \geq \alpha$ for any integral point $y \in \sigma \cap (\Z^n)^*$. A complete rational polyhedral fan $\Sigma$ is \textit{$\alpha$-canonical} if every cone in $\Sigma$ is $\alpha$-canonical. Furthermore, a cone or fan is \textit{canonical} if it is $\alpha$-canonical for $\alpha = 1$. 
\end{definition}

We will now give a characterization of the finiteness of the Fine nef value. We assume $P$ to have the inequality description given by $\langle a_i , x \rangle \geq d_P^F(a_i)$ for $1\leq i \leq m$ in a unique minimal way as in \ref{eq: minimaldescription}. Let $s \geq 0$ and let $v$ be a vertex of $P$ that satisfies with equality the inequalities given by the $a_i$ for $i \in I$ where the other inequalities are strict for $i \notin I$. Note that $\mathcal{N}(P)$ is $\Q$-Gorenstein of index $r$, then using the notation of Definition \ref{Def: gorenstein}, we may define
$$v(s) := v + \frac{s}{r_\sigma}u_\sigma.$$
We have that $v(s)$ is linear as a function of $s$ and that $\langle a_i, v(s) \rangle = d_P^F(a_i) +s$ for $i \in I$.

\begin{theorem} \label{thm: taufinite}
The Fine nef value $\tau^F(P) < \infty$ if and only if $\mathcal{N}(P)$ is $\Q$-Gorenstein and canonical. 
\end{theorem}

\begin{proof} 
\noindent
To see the forward implication, assume that $\tau^F(P) < \infty$. Then there exists some small enough $s > 0$, where without loss of generality we can assume $s \in \Q$, such that $P^{F(s)}$ and $P$ are combinatorially equivalent. Let $v$ a vertex of $P$ and let $v' \in P^{F(s)}$ the vertex corresponding to $v$ under this equivalence. Since $s>0$ then $v'-v \neq 0$. Now, take any linear functional defining a facet incident to $v$, and let this be given by some primitive $a_i \in (\Z^n)^*$. Then we have that
$$\langle a_i , v'-v \rangle =s,$$
which holds for all $a_i$. This shows that $\mathcal{N}(P)$ is $\Q$-Gorenstein. Now, assume that $\mathcal{N}(P)$ is not canonical. Then for some vertex $v$ of $P$ there is a linear functional $a$ such that $a = \sum_{i \in I} \lambda_i a_i$ and $\sum_{i \in I} \lambda_i$ is minimal and strictly less than $1$, where $I$ defines the set of facet defining linear inequalities at $v$. Then for any small $s>0$, it is the case that $a$ will define a facet of $P^{F(s)}$ while it did not define a facet of $P$. Thus $\tau^F(P)$ is infinite.
\bigskip

To see the reverse implication, let us assume that $\mathcal{N}(P)$ is $\Q$-Gorenstein and canonical. Then we can define $v(s)$ for all vertices $v$ of $P$ as in our remarks above. We will show that this implies that for some small $s>0$
\begin{equation} \label{eq: convexhullofvs}
    P^{F(s)} = \operatorname{conv}(v(s) \st v \text{ is a vertex of } P).
\end{equation}

Then, using $v(s) \neq v'(s)$ for $v \neq v'$ and small enough $s$, we obtain a bijection between the vertices of $P$ and $P^{F(s)}$ which preserves incidences with facets. Hence, their face lattices are isomorphic and $\tau^F(P) < \infty$.

For the inclusion $\operatorname{conv}(v(s)) \subseteq P^{F(s)}$, let $\langle a, x \rangle \ge c$ be a valid inequality for $P$. We need to show that for some small $s>0$ and every vertex $v$ of $P$ we have $\langle a, v(s) \rangle \ge c+s$. If $\langle a, v \rangle > c$, any small enough $s$ will do the trick. If $\langle a, v \rangle = c$, then $a$ belongs to the normal cone of $v$. Using the facet defining inequalities $\langle a_i, x \rangle \ge c_i$ which are sharp at $v$, we can write $a = \sum_i \lambda_i a_i$ and $c = \sum_i \lambda_i c_i$ with all $\lambda_i \ge 0$. As we assume $\mathcal{N}(P)$ to be canonical, we have $\sum_i \lambda_i \ge 1$. Hence,
$$\langle a, v(s) \rangle = \sum_{i} \lambda_i \langle a_i, v(s) \rangle = \sum_i \lambda_i (c_i+s) \ge c + s \,.$$

For the other containment, suppose that there is a $w \in P^{F(s)}$ such that $w \notin \operatorname{conv}(v(s))$. Then there exists a linear functional $a \in (\Z^n)^*$ separating $w$ from all the $v(s)$. This $a$ must be contained in the normal cone of some vertex $v$ of $P$. Set
$ c \coloneqq \langle a, v \rangle = \min \{ \langle a, x \rangle \st x \in P\} \,.$
Using the facet defining inequalities $\langle a_i, x \rangle \ge c_i$ which are sharp at $v$, we can write $a = \sum_i \lambda_i a_i$ and $c = \sum_i \lambda_i c_i$ with all $\lambda_i \ge 0$.
The fact that $w \in P^{F(s)}$ translates to the inequalities
$ \langle a_i, w \rangle \ge c_i+s$ for all $i$.

But the fact that $a$ separates $w$ from $\operatorname{conv}(v(s))$ translates to the inequality 
$\langle a, w \rangle < \langle a, v(s) \rangle$
which can be rewritten as
$\sum_i \lambda_i \langle a_i, w \rangle < \sum_i \lambda_i (c_i+s) \,,$
a contradiction.
\end{proof}

\section{Natural Projections in the Fine case}
\label{sec:natproj}

We now want to study the behaviour of the Fine $\Q$-codegree under projections, so we introduce the following definition.

\begin{definition}
Let $K(P)$ be the linear space parallel to $\text{aff}(\core^F(P))$. Then the projection $\pi_P : \R^n \to \R^n/K(P)$ is called the natural projection associated with $P$.
\end{definition}

We now have the following Lemma from \cite{DiR14} which holds with the same proof in the case of Fine Adjunction Theory.

\begin{lemma} \label{lemma:2.2}
Let $x \in \operatorname{relint}(\core^F(P))$. Let us denote by $f_1,...,f_t$ the relevant valid inequalities for $P$ with $d_{f_i}(x) = \mu^F(P)^{-1}$.  Then their primitive inner normals $a_1,..., a_t$ positively span the linear subspace $K(P)^\perp$.\\
Moreover, if $\core^F(P) = \{x\}$, then 
$$\{y \in \R^n \st d_{f_i}(y) \geq 0 \text{ for all } i =1,...,t\}$$
is a rational polytope containing $P$.
\end{lemma}

However, we can prove the following stronger result, which does not hold in the classical polyhedral adjunction case. We include here the proof of the direction that was previously not valid.

\begin{prop} \label{Prop:coreispt}
The image $Q := \pi_P(P)$ of the natural projection of $P$ is a rational polytope satisfying $\mu^F(Q) = \mu^F(P)$. Moreover $\core^F(Q)$ is the point $\pi_P(\core^F(P))$.
\end{prop}

\begin{proof}

To prove that $\mu^F(P)^{-1} \leq \mu(Q)^{-1}$, let $g$ be a valid inequality for $Q$ and let $p \in P$ with $\pi_P(p) = q \in \core^F(Q)$. Then, for $\pi_P^*g = f$, we have
$$\mu^F(Q)^{-1} = g(q) - \min_{\Tilde{q} \in Q} g(\Tilde{q}) = \pi^*g_P(p) - \min_{\Tilde{p} \in P} \pi^*_P g(\Tilde{p}) \geq \mu^F(P)^{-1}$$
which concludes our proof.
\end{proof}

\begin{remark}
Note that we have described the behaviour of the Fine $\Q$-codegree under the natural projection of $P$. However, under any projection $\pi'$ of $P$, we still have that $\mu^F(\pi'(P)) \leq \mu^F(P)$.
\end{remark}

\section{Cayley Decompositions and the Fine structure theorem}
\label{sec: cayleydec}

We let $P \subseteq \R^n$ be an $n$-dimensional lattice polytope and we recall the following definition.

\begin{definition} \label{Def: Cayleysum}
Given lattice polytopes $P_0,...,P_t \subseteq \R^k$, then the \textit{Cayley sum} $P_0\ast \cdots \ast P_t$ is the convex hull of 
$$(P_0 \times 0) \cup (P_1 \times e_1) \cup \cdots \cup (P_t \times e_t) \subseteq \R^k \times \R^t$$
for $e_1,...,e_t$ the standard basis of $R^t$.\\ 
\end{definition}

As a means of comparison, we will now define the notion of codegree which comes up in the context of Ehrhart Theory of lattice polytopes \cite{BN07}.

\begin{definition}
Let $P$ be a rational polytope. We define the \textit{codegree} as
$$\operatorname{cd}(P) := \min \{k \in \N_{\geq 1} \st \operatorname{int}(kP) \cap \Z^n \neq \emptyset \}.$$
\end{definition}

Now, let us define the value
$$d^F(P) :=
\begin{cases}
2(n- \lfloor \mu^F(P) \rfloor), \text{ if } \mu^F(P) \notin \N\\
2(n-\mu^F(P)) + 1, \text{ if } \mu^F(P) \in \N
\end{cases}$$

\noindent
We have that $P \cong \Delta_n$ if and only if $\cd(P) = n+1$. Moreover, $\mu(P) \leq \mu^F(P) \leq \cd(P) \leq n+1$, where this relation is obtained from the original adjoint polytopes case in \cite{DiR14}. Since $\mu^F(\Delta_n) = n+1$, we see that $P \cong \Delta_n$ if and only if $\mu^F(P) = n+1$. \\

\noindent
Hence we come to the following strengthening of the Decomposition Theorem for Cayley Sums \cite[Theorem 3.4]{DiR14} whose proof follows the one presented on \cite{DiR14} slightly adapted to the Fine case.

\begin{theorem} \label{Thm: decomposition}
Let $P$ an $n$-dimensional lattice polytope with $P \ncong \Delta_n$. If $n > d^F(P)$, then $P$ is a Cayley sum of lattice polytopes in $\R^m$ with $m \leq d^F(P)$.
\end{theorem}

\noindent
Let us consider the following example where we compute the codegree in our three settings.

\begin{example} \label{ex: alphacan}
Let $\Delta_n(a) := \operatorname{conv}(0,ae_1,e_2,...,e_n)$ for positive integers $a \in \Z_{>0}$, where the $e_i$ for $1 \leq i \leq n$ form the standard basis of $\R^n$. Note that in the case where $a=1$ this consists of the standard simplex which has been argued before that satisfies
$$\operatorname{cd}(\Delta_n(1)) = \mu(\Delta_n(1)) = \mu^F(\Delta_n(1)) = n+1.$$
Thus, let us consider the case where $a>1$ and $n \geq 2$. It is easy to check that $\operatorname{cd}(\Delta_n(a)) = n$. Moreover, it has been computed in \cite{Paf15} that in this case the $\Q$-codegree is given by
$$\mu(\Delta_n(a)) = n-1+\frac{2}{a}.$$
Finally, since in the Fine case the inequality $\sum_{i=2}^n x_i \leq 1$ is valid, it can be computed that
$$\mu^F(\Delta_n(a)) = n.$$
Thus, we obtain that $\mu(\Delta_n(a)) < \mu^F(\Delta_n(a)) = \operatorname{cd}(\Delta_n(a))$.
\end{example}

\noindent
From this example we see that the $\Q$-codegree and the Fine $\Q$-codegree can take different values on the same polytope $P$.

\section{Finiteness of the Fine $\Q$-codegree spectrum}
\label{sec: finespectrum}

It has been proven already that when bounded from below by some $\varepsilon >0$, the set of values that the $\Q$-codegree can take under certain conditions is finite. We will shortly review these conditions for the case of the original polyhedral adjunction theory.\\

Let $P \subseteq \R^n$ be a lattice polytope of dimension $n$, which we assume to be full-dimensional. 
We define the following sets as in \cite{Paf15},
$$\mathcal{S}(n,\varepsilon) := \{P \st P \text{ is an $n$-dimensional lattice polytope}, \mu(P) \geq \varepsilon  \},$$
$$\mathcal{S}^{can}_\alpha(n,\varepsilon) := \{P \st P \in \mathcal{S}(n,\varepsilon) \text{ and } \mathcal{N}(P) \text{ is $\alpha$-canonical}  \}.$$

The theorem proven in \cite{Paf15} is stated as follows.

\begin{theorem}[Paffenholz, {\cite[Theorem 3.1]{Paf15}}] \label{Thm: finiteqcodspec}
Let $n \in \N$ and $\alpha, \varepsilon > 0$ be given. Then
$$\{\mu(P) \st P \in \mathcal{S}^{can}_\alpha (n,\varepsilon) \}$$
is finite.
\end{theorem}

Note that in this result the $\alpha$-canonical assumption on the polytopes was necessary. 

\begin{example}
A natural example to consider in order to see the importance of this assumption is the family of polytopes
$$\Delta_n(a) = \operatorname{conv}(0,ae_1,...,e_n)$$
where the $e_1,...,e_n$ are the standard basis vectors and $a \in \Z_{>0}$, which was previously studied in Example \ref{ex: alphacan}. For these polytopes, their normal fan is $\Q$-Gorenstein with index $a$ and if $a>1$, then 
$$\mu(\Delta_n(a)) = n-1+\frac{2}{a},$$
but the polytopes $\Delta_n(a)$ are not $\alpha$-canonical for any $\alpha > 0$ and their $\Q$-codegree can take an infinite number of values.
\end{example}

In what follows we will study a generalization of this theorem to the case of Fine adjunction theory. We will follow the proof presented in \cite{Paf15} and adapt it to the Fine polyhedral adjunction case where the remarkable difference will be the fact that we will not be assuming that the polytopes are $\alpha$-canonical, hence in this new setting the theorem holds in much more generality and with much weaker assumptions. We first introduce the following definition.\\

\begin{definition}
A vector $a_i$ is a \textit{Fine core normal} if for all $y \in \operatorname{core}^F(P)$,
$$\langle a_i , y \rangle = d_P^F(a_i) + \mu^F(P)^{-1}.$$
We also define the set
$$\mathcal{S}^F(n,\varepsilon) := \{P \st P \text{ is an $n$-dimensional lattice polytope and } \mu^F(P) \geq \varepsilon  \}.$$
\end{definition}

We can now state our main result.

\begin{theorem} \label{thm:finitespectrum}
Let $n \in \N$ and $\varepsilon > 0$ be given. Then
$$\{\mu^F(P) \st P \in \mathcal{S}^{F}(n,\varepsilon) \}$$
is finite.
\end{theorem}

The proof of our main theorem here will also consist of two main parts. First, we will show that up to lattice equivalence, for a fixed $n \in \Z_{>0}$, there are only finitely many sets of Fine core normals for $n$-dimensional lattice polytopes. Then we will show that each such configuration of core normals gives rise to finitely many values for the Fine $\Q$-codegree above any positive threshold. Thus, if we let $P$ be described by all relevant inequalities as
$$P = \{x \in \R^n \st \langle a_i , x \rangle \geq c_i, i=1,...,m\},$$
note that, up to relabelling, we can assume that the set of Fine core normals consisting of $a_1,...,a_k$ for some $k \leq m$, is the set of valid inequalities defining the affine hull of the Fine core of $P$, $\operatorname{aff}(\core^F(P))$, i.e.,

$$\operatorname{aff}^F(\core(P)) = \{x \st \langle a_i , x \rangle = c_i + \mu^F(P)^{-1}, 1 \leq i \leq k\}.$$

\begin{definition}
Define the set $A_{\core}^F$ to be
$$A_{\core}^F := \operatorname{conv}(a_1,...,a_k) \subseteq (\R^n)^*$$
as the convex hull of the Fine core normals. 
\end{definition}

For $P$ as defined above, the following lemmas will show that all the $a_i$ are vertices of $A_{\core}^F$ and that the origin is a relatively interior point. 

\begin{lemma}[{\cite[Lemma 2.2]{DiR14}}] \label{lemma: origininterior}
The origin is in the relative interior of $A_{\core}^F$.
\end{lemma}

Moreover, the following result proven in \cite{Paf15} gives us precisely the vertices of $A_{\core}^F$.

\begin{lemma}[{\cite[Lemma 3.6]{Paf15}}] \label{lemma: corenormalvertices}
The vertices of $A_{\core}^F$ are $a_1,...,a_k$.
\end{lemma}

We now want to show that, independently of the polytope being $\alpha$-canonical or not, the origin is the only lattice point in the relative interior of $A_{\core}^F$.

\begin{lemma} \label{lemma: originonlylatticept}
For $A_{\core}^F$ as above, it follows that $ \operatorname{relint} (A_{\core}^F ) \cap (\Z^n)^* = \{0\}$.
\end{lemma}

\begin{proof}
We prove this by contradiction. Assume there is some vector $a \in (\Z^n)^* \setminus \{0\}$ contained in the relative interior of $A_{\core}^F$. As $0 \in \operatorname{relint}(A_{\core}^F)$, the point $a$ is contained in the cone spanned by the vertices of some facet $F$ of $A_{\core}^F$, and defines the valid inequality $\langle a, x \rangle \geq b$ for $P$. If we let $a_1,...,a_k$ be the vertices of $A^F_{\core}$, we can find $\lambda_1,...,\lambda_k \geq 0$ with $\lambda_i = 0$ if $a_i \notin F$ such that $a = \sum_{i=1}^k \lambda_i a_i$ and satisfying $\sum_{i=1}^k \lambda_i < 1$.
This last inequality follows from the fact that $a$ is in the relative interior of $A_{\core}^F$. Let $x_{\core} \in \operatorname{relint}(\core^F(P))$. By definition, we have that  $\langle a , x_{\core} \rangle - b \geq (\mu^F(P))^{-1}.$ Now, considering the sum over all valid inequalities associated to the core normals $a_i$ for $1 \leq i \leq k$, for such $y$ we obtain
$$ \langle a ,x_{\core} \rangle -b = \sum_{i=1}^k \lambda_i ( \langle a_i,x_{\core} \rangle - b_i) = \sum_{i=1}^k \lambda_i (\mu^F(P))^{-1} < (\mu^F(P))^{-1} $$
where the last inequality follows from the fact that $\sum_{i=1}^k \lambda_i <1$, but this contradicts the previous relation.
\end{proof}

The last result we need in this first part of the proof is the following one by Lagarias and Ziegler. We say two lattice polytopes $P$ and $Q$ are \textit{lattice equivalent} if there is an affine lattice isomorphism mapping $P$ onto $Q$.

\begin{theorem}[Lagarias, Ziegler {\cite[Theorem 1]{LZ91}}] \label{thm:lagariasziegler}
Let integers $n,l \geq 1$ be given. There are, up to lattice equivalence, only finitely many different lattice polytopes of dimension $d$ with exactly $l$ interior points in the lattice $\Z^n$.
\end{theorem}

Since we have proven that for $A_{\core}^F$ the only lattice point in its relative interior is $\{0\}$, combining this result with Theorem \ref{thm:lagariasziegler} we have shown that for a fixed $n \in \Z_{>0}$, only finitely many sets define the Fine core normals of an $n$-dimensional lattice polytope $P$. We record this as the following result.

\begin{corollary} \label{cor:finitecorenormals}
Let $n \in \Z_{>0}$ be fixed. Then, up to lattice equivalence, only finitely many sets define the Fine core normals of some $n$-dimensional lattice polytope $P$.
\end{corollary}

In what follows, we will continue with the second step of the proof. We make use of the following lemma proven in \cite{Paf15} where we do not require the $\alpha$-canonicity of $P$.

\begin{lemma}[{\cite[Lemma 3.10]{Paf15}}] \label{lemma:finitepolytopesforA}
Fix some $\varepsilon >0$ and some $n \in \Z_{>0}$. Let $P$ be an $n$-dimensional lattice polytope with set of Fine core normals $\mathcal{A}$. Then the set
$$\{\mu^F(P) \st P \text{ is $n$-dimensional with set of Fine core normals $\mathcal{A}$}, \mu^F(P) \geq \varepsilon \}$$
is finite.
\end{lemma}

We have now all the ingredients to prove our main result.

\begin{proof}[Proof of Theorem \ref{thm:finitespectrum}]
Combining this last lemma together with the previous one we obtain the following. First of all, by Corollary \ref{cor:finitecorenormals} we have that up to lattice equivalence, there are only finitely many sets of Fine core normals for some $n$-dimensional lattice polytope. Finally, by Lemma \ref{lemma:finitepolytopesforA} the set of values $\mu^F$ is finite for $n$-dimensional lattice polytopes with a fixed set $\mathcal{A}$ of Fine core normals.
\end{proof}

We have shown that in the Fine case, a version of the theorem regarding the finiteness of the $\Q$-codegree spectrum holds, dropping the $\alpha$-canonicity assumption. Hence, considering all valid inequalities is a condition that highly restricts the shape and properties of the polytope $A_{\core}^F$, i.e., the convex hull of the Fine core normals of a polytope $P$, since all such polytopes contain just one lattice point inside, namely the origin. Due to this we are able to prove the result in greater generality for the Fine $\Q$-codegree spectrum.


\providecommand{\bysame}{\leavevmode\hbox to3em{\hrulefill}\thinspace}
\providecommand{\href}[2]{#2}

\bibliography{bibliography}
\bibliographystyle{alpha}

\end{document}